\documentclass[11pt,a4paper]{article}
\usepackage{epsfig}
\usepackage[numbers]{natbib}
\usepackage{amsmath, amsthm, amssymb}
\usepackage{epsfig}
\usepackage[margin=3cm, top=3.5cm, bottom=3.5cm]{geometry}

\newcommand\blfootnote[1]{%
  \begingroup
  \renewcommand\thefootnote{}\footnote{#1}%
  \addtocounter{footnote}{-1}%
  \endgroup
}
\newtheorem{theorem}{Theorem}[section]
\newtheorem{lemma}{Lemma}[section]

\newtheorem{corollary}{Corollary}[section]

\newtheorem{remark}{Remark}[section]

\numberwithin{equation}{section}
\everymath{\displaystyle}
\allowdisplaybreaks
\title{Application of Hardy inequalities for some singular parabolic equations}
\date{}

\author{
Nikolai Kutev\thanks{Institute of Mathematics and Informatics, Bulgarian Academy of Sciences, 1113, Sofia, Bulgaria}
 \and Tsviatko Rangelov
 \footnotemark[1]
}

\begin{document}

\maketitle
\blfootnote{Corresponding author: T. Rangelov, rangelov@math.bas.bg}

%\today
\begin{abstract}
\noindent
Boundary value problems for non-linear parabolic equations with singular potentials are considered. Existence and non-existence results as an application of different Hardy inequalities are proved. Blow-up conditions are investigated too.
\end{abstract}

\vspace{2pt}


\vspace{2pt}
\noindent

{\bf Math. Subj. Class.}   26D10, 35P15

\section{Introduction}
\label{sec1}
We consider boundary value problem for the singular p-heat equation
\begin{equation}
\label{eq1}
\left\{\begin{array}{ll}
&u_t-\Delta_pu=\mu W(x)|u|^{p-2}u, \ \ \hbox{ for } x\in\Omega, t>0.
\\
&u(x,0)=u_0(x)\ \ \hbox{ for } x\in\Omega,
\\
&u(x,t)=0 \ \ \hbox{ for } x\in\partial\Omega, t>0.
\end{array}\right.
\end{equation}
Here $\Delta_pu=\hbox{div}(|\nabla u|^{p-2}\nabla u)$ is the p-Laplacian, $\mu=const\in \mathbb{R}$, $p>1$, $\Omega$ is a bounded domain in $\mathbb{R}^n$, $n\geq2$, $0\in\Omega$ and function $W(x)$ is singular on $\partial\Omega$ or/and at $0$.

More precisely we consider the following functions $W(x)$
\begin{equation}
\label{eq001}
\begin{array}{ll}
i) &W(x)=d^{-p}(x), d(x)=\hbox{dist}(x,\partial\Omega), \ p>1, p\neq n;
\\
ii)  &W(x)=d^{-p}(x)+\frac{p}{2(p-1)}d^{-p}\log^{-2}\frac{d(x)}{D}, d(x)=\hbox{dist}(x,\partial\Omega), p>n,
\\
&D=\sup_{x\in\Omega}d(x);
\\
iii) &W(x)=\left(|x|\log\frac{R}{|x|}\right)^{-n}  \hbox{ in }  \Omega\subset B_R=\{|x|<R\} ,  R>\hbox{ess sup}_{x\in\Omega}|x| ;
\\
iv) &W(x)=|x|^{m-n}\left|\varphi^m(x)-|x|^m\right|^{-p}  \hbox{ in }  \Omega=\{|x|<\varphi(\theta)\} , p>n , m=\frac{p-n}{p-1},
\\
&\theta=\frac{x}{|x|}  \hbox{ is the angular variable of }  x ,  \hbox{ and } \Omega\ \ \hbox{ is star-shaped with respect }
\\
&\hbox{ to an interior ball centered at the origin }, \varphi\in C^{0,1}.
\end{array}
\end{equation}
The interest to parabolic problems with singular potentials is due to the applications, for example in molecular physics, see \citet{Le67}, quantum cosmology, see \citet{BE97}, quantum mechanics, see \citet{BG84}, electron capture problems, see \citet{GGMS08}, porous medium of fluids, see \citet{AG04} and the combustion models, see \citet{Ge59}. Also in some reaction-diffusion problems involving the heat equation with critical reaction term, the regularized equation is of the type \eqref{eq1}.

In the remarkable pionering paper \citet {BG84} is proved that problem \eqref{eq1} with Hardy potential $W(x)=\frac{1}{|x|^2}$ and $p=2$  has a global solution for $\mu\leq\left(\frac{n-2}{2}\right)^2$. Moreover if $\mu>\left(\frac{n-2}{2}\right)^2$ then the solution of \eqref{eq1} blows-up for a finite time.

The results of Baras and Goldstein are extended in different directions. In \citet{VZ00} the authors present a complete description of the functional framework in which it is possible to obtain well-posedness for the singular heat equation.

The singular parabolic equation
$$
\left\{\begin{array}{ll}
u_t=\Delta u +\frac{\nabla\rho}{\rho}.\nabla u+\frac{\mu}{|x|^2}u, \ \ \hbox{ in } \mathbb{R}^n\times \mathbb{R}^+,
\\
u(x,0)=u_0(x)\geq0, \ \ \hbox{ for } x\in \mathbb{R}^n,
\end{array}\right.
$$
with Kolmogorov operator is studied in \citet{CPT19}, where $\rho$ is probability density on $\mathbb{R}^n$.

The quasilinear case
\begin{equation}
\label{eq01}
\left\{\begin{array}{ll}
u_t-\hbox{div}\left(a(x,t,\nabla u)\right)=\frac{\mu u}{|x|^2}+f(x,t), \ \ \hbox{ in } \Omega\times(0,T),
\\
u(x,0)=u_0(x)\geq0,  \ \ \hbox{ for } x\in \Omega,
\\
u(x,t)=0, \ \ \hbox{ for } (x,t)\in \partial\Omega\times(0,T),
\end{array}\right.
\end{equation}
is considered in \citet{Po19}. In the problem  \eqref{eq01} function $W(x)=|x|^{-2}$ is Hardy potential  and the asymptotic behavior of the solutions of \eqref{eq01} when $f\neq0$ is investigated.

Another generalization of \citet{BG84} is given in \citet{AA98} for problem \eqref{eq1} with general Hardy potential $W(x)=|x|^{-p}$, $1<p<n$, $n\geq3$ and $0\in\Omega$. The authors prove existence and nonexistence results when $\mu\leq\left(\frac{n-p}{p}\right)^p$ or $\mu>\left(\frac{n-p}{p}\right)^p$ respectively. Here $\left(\frac{n-p}{p}\right)^p$ is the optimal Hardy constant in the corresponding Hardy inequality.

In \citet{JYP12} the more general problem
\begin{equation}
\label{eq02}
\left\{\begin{array}{ll}
u_t-\Delta_pu=\frac{\mu}{|x|^s}|u|^{q-2}u, \ \ \hbox{ in } \Omega\times(0,\infty),
\\
u(x,0)=u_0(x)\geq0,  \ \ \hbox{ for } x\in \Omega,
\\
u(x,t)=0, \ \ \hbox{ for } (x,t)\in \partial\Omega\times(0,\infty),
\end{array}\right.
\end{equation}
with $1<p<n$, $0<s\leq p\leq q\leq\frac{n-s}{n-p}p$, $\Omega\subset\mathbb{R}^n$, $0\in\Omega$ is investigated. By means of Sobolev--Hardy inequality
$$
\left(\int_\Omega|\nabla u|^pdx\right)^{q/p}\geq C_{n,p,q,s}\int_\Omega\frac{|u|^q}{|x|^s}dx,\ \ u\in W_0^{1.p}(\Omega),
$$
existence of weak, very weak and entropy solutions to \eqref{eq02} are proved for $\mu\leq C_{n,p,q,s}$ and finite blow-up for $\mu>C_{n,p,q,s}$.

In \citet{AMPT14} where singular reaction--diffusion equation for p-Laplace operator with diffusion and  reaction term
\begin{equation}
\label{eq03}
\left\{\begin{array}{ll}
(u)^\theta_t-\Delta_pu=\frac{\mu|u|^{p-1}}{|x|^p}+u^q+f, \ \ \hbox{ in } \Omega\times(0,T), \ \ u\geq0
\\
u(x,0)=u_0(x)\geq0,  \ \ \hbox{ for } x\in \Omega,
\\
u(x,t)=0, \ \ \hbox{ for } (x,t)\in \partial\Omega\times(0,T),
\end{array}\right.
\end{equation}
is investigated. Here $\theta=1$ or $\theta=p-1$, $f(x)\geq0$, $1<p<n$, $q>0$. The authors prove existence of a critical exponent $q_+(p,\lambda)$ such that for $q<q_+(p,\lambda)$ there are global solutions to \eqref{eq03} and for $q>q_+(p,\lambda)$ there is no solution. In the semilinear case of \eqref{eq03} existence of a critical exponent $q_+(\lambda)$ of Fujita type, see \citet{Fu66} is
stated and existence of a solution is proved if and only if $q<q_+(\lambda)$.

Without a presence of a reaction  term, for $\theta=1$ and $p<\frac{2n}{n+2}$ an existence of a global solution for all $\mu>0$ in \eqref{eq03} is given in \citet{AP00} under suitable conditions on $f$ and $u_0$. The above results are extended to quasilinear parabolic equations without reaction term and with Caffarelli--Kohn--Nirenberg weights in \citet{DGP04}. When $\theta\neq1$, $\mu=0$, the one dimensional case of \eqref{eq03} without reaction term is studied in \citet{EV88}.

Finally, let us mention the results in \citet{BZ16} for the problem
$$
\left\{\begin{array}{ll}
u_t-\Delta u=\frac{\mu}{d^2(x)}u+f, \ \ \hbox{ in } \Omega\times(0,T),
\\
u(x,0)=u_0(x),  \ \ \hbox{ for } x\in \Omega,
\\
u(x,t)=0, \ \ \hbox{ for } (x,t)\in \partial\Omega\times(0,T),
\end{array}\right.
$$
which is singular on $\partial \Omega$. For $\mu\leq\frac{1}{4}$ the authors prove existence of a unique weak solution. If $\mu>\frac{1}{4}$ then there is no control, which means that the blow-up phenomena cannot be prevented. Here $\frac{1}{4}$ is the optimal constant in the corresponding Hardy inequality.

In the present paper we extend the result for \eqref{eq1} with $p=2$ in \citet{BG84} and for $2<p<n$ in \citet{AA98} when the Hardy potential $W(x)=|x|^{-p}$ is singular at the origin, to more general singular parabolic equations with $W(x)$ given in i) - iv) of \eqref{eq001} above. The new Hardy potentials are singular on the boundary of the domain or both at the origin and on the boundary. Let us recall that for the corresponding Hardy inequalities, the Hardy constants are optimal, see \eqref{eq2} - \eqref{eq5}. We prove that for $\mu$ less than the optimal Hardy constant problem \eqref{eq1} has a global weak solution. When $\mu$ is greater than the optimal Hardy constant than the solution blows-up for a finite time. To our best knowledge these type of singular parabolic equations are not consider in the literature.

In   Sect. \ref{sec2}  we recall some Hardy inequalities with weights $W(x)$ and prove global existence of weak solutions of \eqref{eq1} while Sect. \ref{sec3} deals with finite time blow-up of solutions to \eqref{eq1}.
\section{Existence of global solution}
\label{sec2}
In this section we prove  existence of a global, generalized solution of \eqref{eq1} when $W(x)$ has one of the  forms \eqref{eq001}i) - \eqref{eq001}iv)  and $\mu<C_p$ or $\mu<C_{p,n}$ respectively.

For this purpose we recall  some well-known Hardy`s inequalities which are important in the proof of the main results.

Inequality with kernel \eqref{eq001}i)
\begin{equation}
\label{eq2}
\int_\Omega|\nabla u|^pdx\geq C_p\int_\Omega\frac{|u|^p}{d^p(x)}dx,
\end{equation}
for $u\in W^{1,p}_0(\Omega)$, where $C_p=\left(\frac{p-1}{p}\right)^p$, $p>1, n\neq p, n\geq2$,  $d(x)=\hbox{dist}(x,\partial\Omega)$. Here $\Omega$ is a bounded $C^2$ smooth domain in $\mathbb{R}^n$ with nonnegative mean curvature $H(x)\geq0$. The constant $C_p=\left(\frac{p-1}{p}\right)^p$ is optimal. Inequality \eqref{eq2} is proved in \citet{LLL12}, see Theorem 1.2.

Inequality with kernel \eqref{eq001}ii).
\begin{equation}
\label{eq3}
\int_\Omega|\nabla u|^pdx\geq C\int_\Omega\frac{|u|^p}{d^p(x)}\left(1+\frac{1}{2}\left(\frac{p}{p-1}\right)\left(\log^{-2}\frac{d(x)}{D}\right)\right)dx,
\end{equation}
for $ u\in W^{1,p}_0(\Omega)$, where $D\geq\sup_{x\in\Omega}d(x)$,  $p>n\geq2$ and $\Omega\subset \mathbb{R}^n$ is bounded domain with $C^2$ smooth boundary, satisfying some geometric conditions like the mean convexity of $\partial\Omega$. The constant $C_p=\left(\frac{p-1}{p}\right)^p$ is optimal. Inequality \eqref{eq3} is proved in \citet{BFT03a}, see Theorem A.

Inequality with kernel \eqref{eq001}iii).
\begin{equation}
\label{eq4}
\int_\Omega|\nabla u|^ndx\geq C_n\int_\Omega\frac{|u|^n}{\left(|x|\log\frac{R}{|x|}\right)^n}dx,
\end{equation}
for $u\in W^{1,n}_0(\Omega)$, where $R\geq e \sup_{x\in \Omega}|x|$, $\Omega\subset \mathbb{R}^n, n\geq2$, $0\in\Omega$, $\Omega$ is a bounded domain, see Theorem 1.1 in \citet{AE05}.  Inequality \eqref{eq4} for $\Omega=B_1=\{|x|<1\}$  is proved in \citet{II15}. The constant $C_n=\left(\frac{n-1}{n}\right)^n$ is optimal;

Inequality with kernel \eqref{eq001}iv).
\begin{equation}
\label{eq5}
\int_\Omega|\nabla u|^pdx\geq C_{p,n}\int_\Omega\frac{|u|^p}{|x|^{n-m}|\varphi^m(x)-|x|^m|^p}dx,
\end{equation}
for  $u\in W^{1,p}_0(\Omega)$, where $p>n$, $n\geq2$, $\Omega=\{|x|<\varphi(\theta)\}\subset \mathbb{R}^n$, $\theta=\frac{x}{|x|}$ is angular variable of $x$ and $\Omega$  is a star shape domain with respect to a small ball centered at the origin, $\varphi\in C^{0,1}$. The constant   $C_{p,n}=\left(\frac{p-n}{p}\right)^p$ is optimal, see \citet{FKR17} and \citet{KR22}, Theorem 7.2.

By means of Hardy inequalities \eqref{eq2} - \eqref{eq5} we have the following results.
\begin{theorem}
\label{th1}
Suppose $\Omega$ is a bounded $C^2$ smooth domain in $\mathbb{R}^n$, $n\geq2$, with nonnegative mean curvature $H(x)\geq0$ and $p>1, p\neq n$. Then if $\mu<\left(\frac{p-1}{p}\right)^p$, $u_0(x)\in L^2(\Omega)$ problem \eqref{eq1} with $W(x)$ in \eqref{eq001}i) has a global solution
\begin{equation}
\label{eq7}
u(x,t)\in L^\infty\left((0,\tau),L^2(\Omega)\right)\cup L^p\left((0,\tau), W^{1,p}(\Omega)\right), \ \ \hbox{ for all } \tau>0.
\end{equation}
\end{theorem}
\begin{theorem}
\label{th2}
Suppose $\Omega$ is a bounded convex $C^2$ smooth domain in $\mathbb{R}^n$, $n\geq2$. Then if $\mu<\left(\frac{p-1}{p}\right)^p$, $p>1, p\neq n$, $u_0(x)\in L^2(\Omega)$, problem \eqref{eq1} with $W(x)$ defined in \eqref{eq001}ii) has a global solution satisfying \eqref{eq7}.
\end{theorem}
\begin{theorem}
\label{th3}
Suppose $\Omega$ is a bounded domain in $\mathbb{R}^n$, $n\geq2$, $0\in\Omega$. Then if $\mu<\left(\frac{n-1}{n}\right)^n$ , $u_0(x)\in L^2(\Omega)$, problem \eqref{eq1} with $W(x)$ defined in \eqref{eq001}iii) has a global solution satisfying \eqref{eq7}.
\end{theorem}
\begin{theorem}
\label{th4}
Suppose $\Omega=\{|x|<\varphi(x)\}\subset \mathbb{R}^n$, $n\geq2$, is a star-shaped domain with respect to a small ball centered at the origin. Then if $\mu<\left(\frac{p-n}{p}\right)^p$, $p>n$, $u_0(x)\in L^2(\Omega)$, problem \eqref{eq1} with $W(x)$ given by \eqref{eq001}iv) has a global solution, satisfying \eqref{eq7}.
\end{theorem}
\begin{proof}[Proof of Theorems \ref{th1} - \ref{th4}]
In order to prove the Theorems \ref{th1} - \ref{th4} we consider the truncated function
\begin{equation}
\label{eq77}
W_N(x)=\min\{N, W(x)\}, N=1,2,\ldots
\end{equation}
where $W(x)$ is one of the functions \eqref{eq001}i) - \eqref{eq001}iv). We considerbe the solution of the truncated problem
\begin{equation}
\label{eq8}
\left\{\begin{array}{ll}
&u_{N,t}-\Delta_pu_N=\mu W_N(x)|u_N|^{p-2}u_N, \ \ \hbox{ for } x\in\Omega, t>0.
\\
&u_N(x,0)=u_0(x) \ \ \hbox{ for } x\in\Omega,
\\
&u_N(x,t)=0 \ \ \hbox{ for } x\in\partial\Omega, t>0.
\end{array}\right.
\end{equation}
From Proposition 1 in \citet{DGP04}, Theorem 1.2 in \citet{Li69} and the references in \citet{Si87}, problem \eqref{eq8} has a unique solution in distributional sense.

Multiplying \eqref{eq8} with $u_N$ and integrating by parts we get from the Hardy`s inequalities \eqref{eq2} - \eqref{eq5} the following estimates for every $T>0$
\begin{equation}
\label{eq9}
\begin{array}{ll}
&\int_{\Omega}|u_N(x,T)|^2dx+\int_0^T\int_{\Omega}|\nabla u_N(x,t)|^pdxdt
\\
&=\int_{\Omega}|u_0(x)|^2dx+\mu\int_0^T\int_{\Omega}W_N(x)|u_N(x,t)|^pdxdt
\\
&\leq\int_{\Omega}|u_0(x)|^2dx+\mu C^{-1}\int_0^T\int_{\Omega}|\nabla u_N(x,t)|^pdxdt,
\end{array}
\end{equation}
where
 \begin{equation}
 \label{eq99}
 C=\left\{\begin{array}{ll} &C_p=\left(\frac{p-1}{p}\right) \hbox{ in Theorems \ref{th1} and \ref{th2}},
 \\
 &C_n=\left(\frac{n-1}{n}\right) \hbox{ in Theorem \ref{th3}},
 \\
 &C_{n,p}=\left(\frac{p-n}{p}\right) \hbox{ in Theorem \ref{th4}}.
 \end{array}\right.
 \end{equation}
  Since $\mu C^{-1}<1$ we get from \eqref{eq9} the energy estimate
$$
\int_{\Omega}|u_N(x,T)|^2dx+(1-\mu C^{-1})\int_0^T\int_{\Omega}|\nabla u_N(x,t)|^pdxdt
\leq \int_{\Omega}|u_0(x)|^2dx.
$$
From the comparison principle $u_N(x,t)$ is a nondecreasing sequence of functions because $W_N(x,t)\geq W_M(x,t)$ for every $x\in\Omega$, $t>0$ and $N\geq M$. We can pass to the limit $N\rightarrow\infty$ by using Theorem 4.1 in \citet{BM92}. Thus the global solution $u(x,t)$ of \eqref{eq1} is defined as a limit of the solution $u_N(x,t)$ of the truncated problem \eqref{eq8} and $u(x,t)$ has the regularity properties given in \eqref{eq7}.
\end{proof}
\begin{corollary}
\label{cor2-1}
Under the condition of Theorem \ref{th4} if $u_0(x)\in L^\infty(\Omega)$ then $u(x,t)\in L^\infty(\Omega\times[0,T])$.
\end{corollary}
\begin{proof}
Since
$$
\begin{array}{ll}&|\nabla u_N|^{p-2}\nabla u_N.\nabla u_N=|\nabla u_N|^p, \quad ||\nabla u_N|^{p-2}\nabla u_N|=|\nabla u_N|^{p-1},
\\
&|\mu W_N(x)|u_N|^{p-2}u_N|\leq\mu N|u_N|^{p-1}
\end{array}
$$
conditions $(\textbf{B}_1) - (\textbf{B}_4)$ in Theorem 3.2, Chap. 5 in \citet{DiB93} are satisfied with $C_0=C_1=1$, $C_2=\mu N$, $c_i=0, \varphi_i=0, i=0,1,2$ and $\delta=p$.

If $u^+_N(x,t)=\max\{u^N(x,t),0\}$, then $u_N^+$ is a nonnegative weak subsolution of \eqref{eq8} and estimate (3.6) in \citet{DiB93}, Theorem 3.2 becomes
\begin{equation}
\label{eq2-7-1}
\hbox{ess sup}_\Omega u_N^+(.,t)\leq\hbox{ess sup}_\Omega u^+_0(x)+\max \left\{\left(\gamma\int_0^T\int_\Omega(u_N^+(x,t))^pdxdt\right)^{\frac{n}{2p}},1 \right\}
\end{equation}
for some constant $\gamma$ depending on the data.

In the same way $-u_N^-(x,t)=-\min\{u^N(x,t),0\}$ is a nonnegative subsolution of \eqref{eq8} and \eqref{eq2-7-1} gives an estimate from below for $u_N(x,t)$.
\end{proof}
\begin{remark}
\label{rem2-1}
The same estimate holds for more general problem
\begin{equation}
\label{eq2-7-2}
\left\{\begin{array}{ll}
&u_{N,t}-\hbox{div}\left(h_N(x)\nabla u_N|^{p-2}\nabla u_N\right)-\mu \psi_N(x)|u_N|^{p-2}u_N=f(x,t), \ \ \hbox{ in } \Omega\times(0,T),
\\
&u_N(0,t)=0 \ \ \hbox{ for } (x,t)\in\partial\Omega\times(0,T),
\\
&u_N(x,0)=u_0(x) \ \ \hbox{ for } x\in\Omega.
\end{array}\right.
\end{equation}
where $u_0(x)\in L^\infty(\Omega)$.

In this case
\begin{equation}
\label{eq2-7-3}\begin{array}{l}
h_N(x)=\min\{|x|^l,N\} \hbox{ for }  l<0,  h_N(x)=1 \hbox{ for }  l=0,   \hbox{ and }  h_N(x)=|x|^l+\frac{1}{N} \hbox{ for }  l>0 ,
\\
\psi_N(x)=\left(|x|^{p-l}+\frac{1}{N}\right)^{-1}\left(1+\frac{1}{N}-\left(\frac{|x|}{\varphi(x)}\right)^{\frac{l+n-p}{p}}\right)^{-p}
\end{array}
\end{equation}
Here  $\Omega\subset\mathbb{R}^n$, $n\geq2$, $0\in\Omega$, is a star-shaped domain with respect to a small ball centered at the origin,   $l+n>p>1$, $l<p$, see for more details Chap. 9.2 in \citet{KR22}. The only difference are the constants $C_0=\frac{1}{N}$, $C_1=\max\{N,(\hbox{ diam}\Omega)^l+1\}$, $C_2=\nu N^{p+1}$, $c_i=0, \varphi_i=0, i=0,1,2$,$\delta=p$.
\end{remark}

%%%%%%%%%%%%
\section{Finite time blow-up}
\label{sec3}
In this section we prove finite time blow-up
of the solutions to the problem \eqref{eq1} with $u_0(x)>0$ for $x\in\bar{\Omega}$, i.e.,
\begin{equation}
\label{eq35}
\left\{\begin{array}{ll}
&u_t-\Delta_pu=\mu W(x)|u|^{p-2}u, \ \ \hbox{ for } x\in\Omega, t>0.
\\
&u(x,0)=u_0(x)>0 \ \ \hbox{ for } x\in\bar{\Omega},
\\
&u(x,t)=0 \ \ \hbox{ for } x\in\partial\Omega, t>0,
\end{array}\right.
\end{equation}
for $W(x)$ defined in \eqref{eq001}i) - \eqref{eq001}iv).

If $\lambda_{1N}$, $\phi_{1N}(x)$ are the first eigenvalue and the first eigenfunction of the problem
$$
\left\{\begin{array}{ll}
&-\Delta_p\phi_{1N}=\lambda_{1N}W_N(x)|\phi_{1N}|^{p-2}\phi_{1N}, \ \ x\in\Omega,
\\
&\phi_{1N}=0, \ \ x\in \partial\Omega
\end{array}\right.
$$
where $W_N(x)$ is defined in\eqref{eq77}, then the following result holds.
\begin{lemma}
\label{lem1}
Suppose $\Omega$ is a bounded smooth domain in $\mathbb{R}^n$, $n\geq2$ and $\lambda>C$, where  $C$ is given  in \eqref{eq99}. Then we have
\begin{equation}
\label{eq37}
\lambda_{1N}\geq C, \ \ N=1,2,\ldots
\end{equation}
\begin{equation}
\label{eq38}
\lim_{N\rightarrow\infty}\lambda_{1N}=C
\end{equation}
\end{lemma}
\begin{proof}
Inequality \eqref{eq37} follows from  the Rayleigh quotient and Hardy inequalities \eqref{eq2} - \eqref{eq5}. Indeed
$$
\lambda_{1N}=\inf_{v\in W_0^{1,p}(\Omega)}\frac{\int_\Omega|\nabla v|^pdx}{\int_\Omega W_N(x)|v|^pdx}
\geq \inf_{v\in W_0^{1,p}(\Omega)}\frac{\int_\Omega|\nabla v|^pdx}{\int_\Omega W(x)|v|^pdx}\geq C
$$
In order to proof \eqref{eq38} we suppose by contradiction that
\begin{equation}
\label{eq39}
\lim_{N\rightarrow\infty}\lambda_{1N}=C+2\delta, \hbox{ for some } \delta>0.
\end{equation}
From the optimality of the constant $C$ in Hardy inequalities \eqref{eq2} - \eqref{eq5} there exists $v_N\in W_0^{1,p}(\Omega)$ such that
$$
\frac{\int_\Omega|\nabla v_N|^pdx}{\int_\Omega W_N(x)|v_N|^pdx}<C+\delta
$$
Thus $\lambda_{1N}\leq C+\delta$ for every $N=1,2,\ldots$ which contradicts  \eqref{eq39}.
\end{proof}
We will construct a positive sub-solution to \eqref{eq35} by the method of separation of variables, i.e., a sub-solution $v(x,t)=X(x)T(t)$ where $X(x)$ and $T(t)$ satisfy the problems.
\begin{equation}
\label{eq40}
\left\{\begin{array}{ll}
&-\Delta_pX(x)-\mu W_N(x)X^{p-1}(x)=- X(x), \ \ x\in \Omega,
\\
&X(x)=0, \ \ x\in\partial\Omega.
\end{array}\right.
\end{equation}
\begin{equation}
\label{eq41}
\left\{\begin{array}{ll}
&T'(t)= T^{p-1}(t), \ \ t>0,
\\
&T(0)=T_0.
\end{array}\right.
\end{equation}
Since $T(t)=T_0\left[1-(p-2) T_0^{p-2}t\right]^{-\frac{1}{p-2}}$ for $p>2$, then  $T(t)$ blows-up for $t=\left[(p-2)T^{p-2}_0\right]^{-1}$.

For $a^{p-2}=\mu$, after the change of the functions $aX=Y$, problem \eqref{eq40} becomes
\begin{equation}
\label{eq42}
\left\{\begin{array}{ll}
&-\Delta_pY(x)=\mu \left(W_N(x)Y^{p-1}(x)-Y(x)\right), \ \ x\in \Omega,
\\
&Y(x)=0, \ \ x\in\partial\Omega.
\end{array}\right.
\end{equation}
We need the following lemma (see Lemma 5.1 in \citet{AA98})
\begin{lemma}
\label{lem2}
(1) There exists a constant $A>0$ such that if \eqref{eq42} has a positive solution $Y(x)$ then $\|Y\|_{L^\infty(\Omega)}>A$;

(2) If $\lambda_{1N}$ is the first eigenvalue of $-\Delta_p$ with weight $W_N(x)$ then $\lambda_{1N}$ is the unique bifurcation point from infinity for the problem \eqref{eq42}
\end{lemma}

As it is mention in  \citet{AA98}  the unbounded bifurcation branch of positive solutions given in Lemma \ref{lem2}(2) cannot cross the level $\|Y\|_{L^\infty(\Omega)}=A$ from Lemma \ref{lem2}(1), neither, obviously, the hyperplane $\mu=0$ for $\mu>\lambda_{1N}$. Hence, problem \eqref{eq42} has at least one positive solution for $\mu>\lambda_{1N}$, i.e., from \eqref{eq38} for $N>N_0$, $N_0$ sufficiently large.
\begin{theorem}[Comparison principle]
\label{th53}
Suppose $u(x,t)$, defined in \eqref{eq7} is solution of \eqref{eq8} and $v(x,t)$ is a positive sub-solution to \eqref{eq8}, i.e.,
\begin{equation}
\label{eq49}
\left\{\begin{array}{ll}
&v_t-\Delta_pv\leq\lambda W_N(x)v^{p-1}, \ \ \hbox{ for } x\in\Omega, t\in[0,\tau),
\\
&v(x,0)\leq u_0(x), \ \ u_0(x)>0 \ \ \hbox{ for } x\in\bar{\Omega},
\\
&v(x,t)\leq0 \ \ \hbox{ for } x\in\partial\Omega, t\in[0,\tau),
\end{array}\right.
\end{equation}
for a fixed $\tau$. If
\begin{equation}
\label{eq491}
\|u(.,t)\|_{L^\infty(\Omega)}\leq M<\infty , \|v(.,t)\|_{L^\infty(\Omega)}\leq M<\infty
\ \ \hbox{ for } t\geq0,
\end{equation}
then
\begin{equation}
\label{eq50}
v(x,t)\leq u(x,t), \ \ \hbox{ for } x\in\Omega, t\in[0,\tau).
\end{equation}
\end{theorem}
\begin{proof}
Suppose by contradiction that \eqref{eq50} fails, i.e., the set $Q=\Omega_+(s)\times[0,\tau)$, $\Omega_+(s)=\{x\in\Omega, s\in[0,\tau);  v(x,s)>u(x,s)\}$ is not empty.
We change the variables $v(t,x)=e^{Kt}z(t,x)$, $u(t,x)=e^{Kt}w(t,x)$ where
\begin{equation}
\label{eq492}
K=\lambda N(p-1)(2M)^{p-2}.
\end{equation}
Multiplying \eqref{eq49} and \eqref{eq8} after the change of variables with $(z-w)_+$ where
$$
(z-w)_+=\left\{\begin{array}{l}z-w, \quad \hbox{ for } z>w,\\0, \quad \hbox{ for } z\leq w\end{array}\right., (z-w)_+\in W_0^{1,p}(\Omega\times(0,\tau)),
$$
 and integrating in $(0,t)\times\Omega, t<\tau$ we get for their difference
\begin{equation}
\label{eq51}
\begin{array}{ll}
0\geq&\int_0^t\int_{\Omega_+(s)}e^{Ks}(z-w)_+(z-w)_tdsdx
\\
&-\int_0^t\int_{\Omega_+(s)}e^{(p-1)Ks}(z-w)\hbox{div}\left[|\nabla z|^{p-2}\nabla z-|\nabla w|^{p-2}\nabla w\right](z-w)dxds
\\
&+\int_0^t\int_{\Omega_+(s)}Ke^{Ks}(z-w)^2-\lambda\int_0^t\int_{\Omega_+(s)}e^{(p-1)Ks}W_N(x)\left[|z|^{p-2}z-|w|^{p-2}w\right](z-w)dxds
\\
&=\frac{1}{2}\int_{\Omega_+(t)}\left[z(x,t)-w(x,t)\right]^2e^{Kt}dx-\frac{1}{2}\int_{\Omega_+(0)}\left[z(x,0)-w(x,0)\right]^2dx
\\
&+\int_0^t\int_{\Omega_+(s)}\left[Ke^{Ks}(z-w)^2-\lambda e^{(p-1)Ks}W_N(x)(|z|^{p-2}z-|w|^{p-2}w)\right](z-w)dsdx
\\
&+\int_0^t\int_{\Omega_+(s)}e^{(p-1)Ks}\left[|\nabla z|^{p-2}z_{x_i}(z_{x_i}-w_{x_i})-|w|^{p-2}w_{x_i}(z_{x_i}-w_{x_i})\right]dsdx
\\
&=\frac{1}{2}\int_{\Omega_+(t)}(z(x,t)-w(x,t)^2dx-\frac{1}{2}\int_{\Omega_+(0)}(z(x,0)-w(x,0)^2dx
\\
&+\int_0^t\int_{\Omega_+(s)}\left[(z-w)^2K-\lambda W_N(x)(z-w)(|z|^{p-2}z-|w|^{p-2}w)\right]dxds
\\
&+\int_0^t\int_{\Omega_+(t)}\left[|\nabla z|^{p-2}z_{x_i}(z_{x_i}-w_{x_i})-|\nabla w|^{p-2}|w_{x_i}(z_{x_i}-w_{x_i})\right]dsdt
\end{array}
\end{equation}
Simple computation give us
\begin{equation}
\label{eq52}
\begin{array}{ll}
&|\nabla z|^{p-2}z_{x_i}(z_{x_i}-w_{x_i})-|\nabla w|^{p-2}u_{x_i}(z_{x_i}-w_{x_i})
\\
&=|\nabla z|^p-|\nabla z|^{p-2}z_{x_i}w_{x_i}+|\nabla w|^p-|\nabla w|^{p-2}z_{x_i}w_{x_i}
\\
&=\frac{1}{2}|\nabla z|^p+\frac{1}{2}|\nabla z|^{p-2}|\nabla(z-w)|^2+\frac{1}{2}|\nabla w|^p
\\
&+\frac{1}{2}|\nabla w|^{p-2}|\nabla(z-w)|^2-\frac{1}{2}|\nabla z|^{p-2}|\nabla w|^2-\frac{1}{2}|\nabla w|^{p-2}|\nabla z|^2
\\
&=\frac{1}{2}\left[|\nabla z|^{p-2}+|\nabla w|^{p-2}\right]|\nabla(w-z)|^2
\\
&+\frac{1}{2}\left[|\nabla z|^{p-2}-|\nabla w|^{p-2}\right]\left[|\nabla z|^2-|\nabla w|^2\right]\geq0.
\end{array}
\end{equation}
Since
$$
\int_{\Omega_+(0)}\left[z(x,0)-w(x,0)\right]^2dx=0,
$$
and
\begin{equation}
\label{eq493}
\begin{array}{ll}
&0\leq(|z|^{p-2}z-|w|^{p-2}w)( z-w)=(p-1)(z-w)^2\int_0^1|(1-\theta)z+\theta w|^{p-2}d\theta
\\
&\leq (p-1)(z-w)^2(|z|+|w|)^{p-2}\leq(p-1)(2M)^{p-2}(z-w)^2
\end{array}
\end{equation}
from \eqref{eq492} - \eqref{eq493} we get for every $t\in[0,\tau)$
$$
\int_{\Omega_+(t)}\left[v(x,t)-u(x,t)\right]^2dx=0.
$$
Thus $v(x,t)=u(x,t)$ for every $x\in\Omega_+(t)$ and every $t\in[0,\tau)$, i.e., for $(x,t)\in Q$ and $Q=\emptyset$ which contradicts our assumption.
\end{proof}
\begin{corollary}
\label{cor3-1}
Suppose  $u(x,t)$ is a solution of \eqref{eq2-7-2} and $v(x,t)$ is a positive subsolution to \eqref{eq2-7-2}, i.e.,
\begin{equation}
\label{eq2-7-4}
\left\{\begin{array}{ll}
&v_t-\hbox{div}\left(h_N(x)\nabla v|^{p-2}\nabla v\right)-\mu \psi_N(x)|v|^{p-2}v\leq f(x,t), \ \ \hbox{ in } \Omega\times(0,T),
\\
&v(0,t)=0 \ \ \hbox{ for } (x,t)\in\partial\Omega\times[0,T],
\\
&v(x,0)\leq u_0(x), u_0(x)>0 \ \ \hbox{ for } x\in\Omega.
\end{array}\right.
\end{equation}
If \eqref{eq2-7-3} holds and $\|u\|_{L^\infty}(\Omega)\leq M$, $\|v\|_{L^\infty}(\Omega)\leq M$ for every $t\in[0,T]$ then
\begin{equation}
\label{eq2-7-5}
v(x,t)\leq u(x,t), \quad \hbox{for } (x,t)\in \partial\Omega\times[0,T].
\end{equation}
\end{corollary}
The proof of Corollary \ref{cor3-1} is identical  with the proof of Theorem \ref{th53}.

The following theorem is the main result in this section.
\begin{theorem}
\label{th52}
Suppose $p>n\geq2$, $u_0\in L^\infty(\Omega)$, $u_0(x)>0$ in $\bar{\Omega}$ and $\mu>C$, where $C$ is given in \eqref{eq99}. Then the solution of \eqref{eq8} blows-up for a finite time.
\end{theorem}
\begin{proof}
Let $X(x)$ be the positive solution of \eqref{eq40} according to Lemma \ref{lem2} and $T(t)$ be the solution of \eqref{eq41} with $T_0=\varepsilon$, $\varepsilon$ small enough such that
\begin{equation}
\label{eq46}
\varepsilon X(x)\leq u_0(x), \ \ \hbox{for } x\in \Omega.
\end{equation}
Thus $v(x,t)=X(x)T(t)$ is a positive solution of the problem
\begin{equation}
\label{eq47}
\left\{\begin{array}{ll}
&v_t-\Delta_pv=\mu W_Nv^{p-1}, \ \ \hbox{ for } x\in\Omega, t\in(0,T_{max}),
\\
&v(x,0)=\varepsilon X(x)\geq0 \ \ \hbox{ for } x\in\Omega,
\\
&v(x,t)=0 \ \ \hbox{ for } x\in\partial\Omega, t\in(0,T_{max}).
\end{array}\right.
\end{equation}
Here $T_{max}$ is defined as
$$
T_{max}=\left[(p-2)\varepsilon^{p-2}\right]^{-1},
$$
and $[0,T_{max})$ is the maximal existence time interval for the solution of \eqref{eq47}.

From  \eqref{eq46} it follows that $v(x,t)$ is a positive sub-solution to \eqref{eq8}.

According the comparison principle, Theorem \ref{th53}, we get
\begin{equation}
\label{eq48}\begin{array}{ll}
&u_N(x,t)\geq X(x)T(t)=\varepsilon X(x)\left[1-(p-2)\varepsilon^{p-2}t\right]^{-\frac{1}{p-2}}\rightarrow\infty
\\
&\hbox{for } t\rightarrow\left[(p-2)\varepsilon^{p-2}\right]^{-1}.
\end{array}
\end{equation}
Since the solutions of \eqref{eq1} are defined as the limit of the solutions $u_N(x,t)$ of the truncated problem \eqref{eq8}, then from \eqref{eq48} it follows that $u(x,t)$ blows-up for a finite time.
\end{proof}

\end{document}